%% file: 20251008_Power_Series_Rings_Sayanagi.tex
\documentclass[11pt]{amsart}
\usepackage[margin=1in]{geometry}
\usepackage{graphicx, enumerate, amsmath, amssymb, xspace, color, cancel, mathtools, paralist, caption, subcaption, wrapfig, framed, pgfplots, tikz-cd, stmaryrd, setspace}

\usepackage{multirow}
\usepackage[table]{xcolor}

\usepgfplotslibrary{fillbetween}

\DeclareMathOperator{\im}{Im}

\DeclareMathOperator{\Ext}{Ext}

\DeclareMathOperator{\spec}{Spec}
\DeclareMathOperator{\Ht}{ht}

\DeclareMathOperator{\cgrade}{c.grade}

\usepackage{amsthm}
\newtheorem{thm}{Theorem}[section]

\newtheorem{de}[thm]{Definition}
\newtheorem{prop}[thm]{Proposition}
\newtheorem{cor}[thm]{Corollary}
\newtheorem{lemma}[thm]{Lemma}
\theoremstyle{definition}
\newtheorem{rem}[thm]{Remark}
\newtheorem{ex}[thm]{Example}

\begin{document}

%
%
\title{Power Series Rings over Zero-Dimensional Rings}

\author{M. Richard Sayanagi}
\address{Department of Mathematical Sciences, New Mexico State University, 
Las Cruces, NM 88003}
\email{mrsayana@nmsu.edu}

\begin{abstract}
A power series ring over non-Noetherian rings can fail to be flat over the base ring, and its dimension can be infinite, even when the dimension of the base ring is finite. We study the case when the base ring has Krull dimension 0, and consider a version of the power series ring which preserves flatness and whose dimension remains finite. We consider properties that are well-understood in the Noetherian context, such as Krull's Height Theorem and unmixedness, and generalize them to the non-Noetherian setting.
\end{abstract}


\maketitle

\spacing{1.08}

 

\section{Introduction}

All rings are assumed to be commutative with identity. All subrings of a ring $R$ are assumed to share the same identity element as the identity of $R$. The dimension of a ring is taken to be the Krull dimension of the ring, unless specified otherwise.

Given a ring $R$, power series rings over $R$ can exhibit behavior that is not observed in polynomial rings over $R$. If $\dim R = m$, then $m + 1 \leq \dim R[X] \leq 2m +1$ \cite[Theorem 2]{Seidenberg}, so the dimension of a polynomial ring remains finite when the base ring has finite dimension. In particular, $\dim R[X]= m + 1$ if $R$ is Noetherian. For a power series ring over $R$, if $\dim R = m$ and $R$ is Noetherian, then $\dim R \llbracket X \rrbracket = m + 1$, but in 1973, Arnold \cite{Arnold1973} showed that  the dimension of a power series ring over a non-Noetherian ring with finite dimension can be infinite. Arnold provided a necessary and sufficient condition for the dimension of a power series ring over a 0-dimensional ring to be finite, a near-Noetherian property called the SFT (strong finite type)  property. An ideal $I$ in a ring $R$ is called \emph{SFT} if there exists $k \in \mathbb{N}$ and a finitely generated ideal $J \subseteq I$ such that $a^k \in J$ for all $a \in I$, and $R$ is called an \emph{SFT ring} if every ideal of $I$ is SFT. Arnold's paper sparked a series of studies \cite{ccd, coykendall, loper} into the dimension question for power series rings. 

Moreover, while $R[X]$ is flat over $R$ for any ring $R$, $R \llbracket X \rrbracket$ may fail to be flat over $R$ if $R$ is a non-Noetherian ring. A sufficient condition for the power series ring to be flat over the base ring was given by Chase \cite{chase} in 1960, involving the idea of coherence. A ring $R$ is said to be \emph{coherent} if every finitely generated ideal in $R$ is finitely presented. $R$ is a coherent ring if and only if every direct product of flat $R$-modules is flat \cite[Theorem 2.1]{chase}, and it can be shown as a consequence that if $R$ is at most countably infinite, then $R \llbracket X \rrbracket$ is faithfully flat over $R$ if and only if $R$ is coherent \cite[Exercise 19.26]{anderson}.

%

One goal of this paper is to construct a version of the power series ring over a non-Noetherian base ring that exhibits dimension and flatness properties seen in the Noetherian context. We present a version of the power series ring that is a subring of the power series ring over a 0-dimensional ring which is in some sense ``large" inside the power series ring, but retains some of the properties seen in power series rings over Noetherian rings. To do so, we use a characterization of 0-dimensional rings given by Gilmer and Heinzer \cite{GH1992}, which provides a way to represent a 0-dimensional ring as a direct limit of Artinian subrings with some integrality conditions. 

The version of the power series ring we construct is non-Noetherian if the base ring is non-Noetherian, but we observe Cohen-Macaulay type properties, such as unmixedness and the ``grade = height" property. We recall that the formal power series ring over Artinian rings are Cohen-Macaulay \cite[Theorem 23.5]{matsumura}. In particular, if $R$ is an Artinian ring, then $R \llbracket X_1, \ldots, X_n \rrbracket$ is a Cohen-Macaulay ring of dimension $n$. Glaz \cite{glaz} introduced ways to characterize Cohen-Macaulayness in the non-Noetherian context, and several later papers \cite{at, H2, HM} have proposed alternative ways to define non-Noetherian Cohen-Macaulayness. In \cite[Remark 5.9]{olb1}, it is shown that complete intersections over 0-dimensional rings satisfy all seven notions of Cohen-Macaulayness in \cite{at}. We discuss the ways in which our version of the power series ring satisfies the different notions of Cohen-Macaulayness in the non-Noetherian context. 


%


\section{Definitions}


Given any ring $S$, we can write $S$ as a directed union of Noetherian subrings. More precisely, we can identify its prime subring $\pi$, which is isomorphic to either $\mathbb{Z}$ or $\mathbb{Z}/n\mathbb{Z}$ for some $n \in \mathbb{N}$, hence Noetherian, then we consider the family of all subrings $S_\alpha \subseteq S$ that are finitely generated over $\pi$. Each $S_\alpha$ is Noetherian, and $\{S_\alpha\}$ is a directed family of Noetherian subrings such that $S = \bigcup_\alpha S_\alpha$. This description of the ring $S$, however, does not give information about the dimension of each $S_\alpha$ in relation to the dimension of $S$, nor can we infer integrality conditions on the extension $S_\alpha \subseteq S$.


For a 0-dimensional ring $R$, Gilmer and Heinzer \cite{GH1992} provide a way to write $R$ as a directed union of Artinian subrings with integrality conditions. 
We will not assume in this paper that a local or semilocal ring is Noetherian.

\begin{prop} \cite[Corollary 5.5]{GH1992} \label{gilmerheinzer}
	Let $R$ be a 0-dimensional ring. If $R$ is integral over an Artinian subring, then there is a directed family $\{R_\alpha\}$ of Artinian subrings such that $R$ is integral over $R_\alpha$ for each $\alpha$, and $R = \bigcup_\alpha R_\alpha$.  In particular, a semilocal 0-dimensional ring is integral over some Artinian subring. Thus, a semilocal 0-dimensional ring can be written as a directed union of Artinian subrings with this integrality condition. 
\end{prop}

It will be convenient to be able to refer to an Artinian subring $R_\alpha$ of a 0-dimensional ring $R$ such that $R_\alpha \subseteq R$ is an integral extension. We define such Artinian subrings here.

\begin{de} 
	Let $R$ be a 0-dimensional ring. We say that an Artinian subring $R_0 \subseteq R$ is \emph{large} in $R$ if $R_0 \subseteq R$ is an integral extension.
\end{de}

If $R$ is a semilocal ring, Gilmer and Heinzer \cite[Theorem 5.4, Corollary 5.5]{GH1992} give a canonical way to build a directed family of large Artinian subrings of $R$ so that their directed union is $R$. This involves the construction of an Artinian subring $R_0$ that is large in $R$, and considering the family of all subrings $R$ that are finitely generated over $R_0$. We introduce some notation associated with this construction of a directed family of subrings of $R$.

\begin{de}
	Let $R \subseteq S$ be an integral extension of 0-dimensional rings. Then $(R, S)$ is called a 0-dimensional pair \cite[Corollary 4.2]{GH1992}. Given a 0-dimensional pair $(R, S)$, let $\mathcal{F}(R,S)$ be the family of all subrings of $S$ that are finitely generated over $R$, i.e.
	\[
		\mathcal{F}(R,S) = \{ S_\alpha \subseteq S \, | \, S_\alpha \text{ is finitely generated over }R\}.
	\]
\end{de}

Observe that $\mathcal{F}(R,S)$ is a directed family of subrings such that the directed union $\bigcup_\alpha S_\alpha$ is equal to $S$. We will simply write $\mathcal{F}$ for $\mathcal{F}(R,S)$ when the context is clear. We note that for each $\alpha$, $R \subseteq S_\alpha$ is an integral ring extension, so $S_\alpha$ can be viewed as a finitely generated $R$-module. Moreover, each subring $S_\alpha \in \mathcal{F}$ is an Artinian ring that is large in $S$. So $\mathcal{F}$ is a directed family of large Artinian subrings of $S$ such that $\bigcup_\alpha S_\alpha = S$.


We consider how to form power series rings in accordance with this directed union construction of a 0-dimensional ring $R$ in terms of large Artinian subrings.

\begin{lemma} \label{directedunion}
	Let $R$ be 0-dimensional ring, and suppose there is an Artinian subring $R_0$ that is large in $R$. Write $\mathcal{F}$ for the family $\mathcal{F}(R_0, R) = \{R_\alpha\}$, and consider the set
	\[
		S = \Big\{ f \in R \llbracket X_1, \ldots, X_n \rrbracket \ \Big| \ 
		\text{there is some } R_\alpha \in \mathcal{F} \text{ such that all coefficients of } f \text{ are in  } R_\alpha \Big\}.
	\]
	Then $S$ is a ring, and is equal to the directed union $\displaystyle{\bigcup_\alpha (R_\alpha \llbracket X_1 \ldots, X_n \rrbracket)}$.
\end{lemma}
\begin{proof}
	We show $S =  \displaystyle{\bigcup_\alpha (R_\alpha \llbracket X_1 \ldots, X_n \rrbracket)}$. If $f \in S$, all coefficients of $f$ are in some $R_{\alpha_f} \in \mathcal{F}$, so
	\[
		 f \in R_{\alpha_f} \llbracket X_1 \ldots, X_n \rrbracket \subseteq \bigcup_\alpha (R_\alpha \llbracket X_1 \ldots, X_n \rrbracket).
	\]
	Next, let $g \in  \bigcup_\alpha (R_\alpha \llbracket X_1 \ldots, X_n \rrbracket)$. Then there is some $R_{\alpha_g} \in \mathcal{F}$ such that $g \in R_{\alpha_g} \llbracket X_1 \ldots, X_n \rrbracket$. This means all coefficients of $g$ are in $R_{\alpha_g}$, so that $g \in S$.

	The directed union of rings is a ring, so $S =\displaystyle{\bigcup_\alpha (R_\alpha \llbracket X_1 \ldots, X_n \rrbracket)}$ is a ring.
\end{proof}

We apply Lemma \ref{directedunion} to define the \emph{ring of Artinian power series}, our main object of study.

\begin{de}\label{artpow}
	Let $R$ be a 0-dimensional ring, and suppose there is an Artinian subring $R_0$ that is large in $R$. Write $\mathcal{F}$ for the family $\mathcal{F}(R_0, R) = \{R_\alpha\}$. If $f \in R\llbracket X_1, \ldots, X_n \rrbracket$ is a power series and all coefficients of $f$ are contained in some $R_\alpha \in \mathcal{F}$, we say that $f$ is an \emph{Artinian power series}. Consider the set of all Artinian power series,
	\[
		\mathcal{F}\llbracket X_1, \ldots, X_n \rrbracket = \big\{ f \in R \llbracket X_1, \ldots, X_n \rrbracket \ \big| \ f \text{ is an Artinian power series} \big\}.
	\]
	We call $\mathcal{F}\llbracket X_1, \ldots, X_n \rrbracket$ the ring of Artinian power series over the pair $(R_0, R)$ in $n$ indeterminates. By Lemma \ref{directedunion}, we have
	\[
		\mathcal{F}\llbracket X_1, \ldots, X_n \rrbracket  = \bigcup_\alpha \big( R_\alpha \llbracket X_1, \ldots, X_n \rrbracket \big)
	\]
	so that $\mathcal{F}\llbracket X_1, \ldots, X_n \rrbracket$ is a directed union of power series rings.
\end{de}

Since any 0-dimensional semilocal ring $R$ is integral over some Artinian subring $R_0$ by Proposition \ref{gilmerheinzer}, the ring of Artinian power series can always be defined for a 0-dimensional semilocal ring. 

\begin{rem}
	Let $R$ be a 0-dimensional ring that is integral over an Artinian subring $R_0$. Write $\mathcal{F} = \mathcal{F}(R_0, R)$. We note that 
	\[
		R \subseteq R[X_1, \ldots, X_n] \subseteq \mathcal{F}\llbracket X_1, \ldots, X_n\rrbracket \subseteq R \llbracket X_1, \ldots, X_n \rrbracket.
	\]
	If $R$ is non-Noetherian, then $S = \mathcal{F}\llbracket X_1, \ldots, X_n\rrbracket$ is non-Noetherian also. Indeed, let $I$ be an ideal of $R$ that is not finitely generated, so there is a countably infinite subset $\{r_i\}_{i =0}^\infty \subseteq I$ such that $r_{i+1} \notin (r_0, \ldots, r_i)R$. Then we have an increasing chain of ideals
	\[
		(r_0)R \subsetneq (r_0, r_1)R \subsetneq \cdots \subsetneq (r_0, r_1, \ldots, r_i)R \subsetneq \cdots 
	\]
	that does not stabilize. Extending these ideals to $S$, we claim
	\[
		(r_0, \ldots, r_i) S \subsetneq (r_0, \ldots, r_i, r_{i+1})S
	\]
	for each $i$. If $(r_0, \ldots, r_i) S = (r_0, \ldots, r_i, r_{i+1}) S$ for some $i$, then $r_{i+1} \in (r_0, \ldots, r_i) S$, so there exist $f_0, \ldots, f_i \in S$ with $r_{i+1} = r_0 f_0 + \cdots + r_i f_i$. Comparing the constant terms in the equation, we have $r_{i+1} = r_0 s_0 + \cdots + r_i s_i$ for some $s_0, \ldots, s_i \in R$, implying $r_{i+1} \in (r_0, \ldots, r_i)R$, a contradiction. Thus, we have an increasing chain of ideals
	\[
		(r_0) S \subsetneq (r_0, r_1) S \subsetneq \cdots \subsetneq (r_0, r_1, \ldots, r_i) S \subsetneq \cdots 
	\]
	that does not stabilize in $S$, so $S$ is not Noetherian.
	
	Moreover, $S$ is strictly smaller than $R \llbracket X_1, \ldots, X_n \rrbracket$. Consider a power series $f(X) \in R\llbracket X \rrbracket$ in one interdeminate, defined by
\[
	f(X) = r_0 + r_1 X + r_2 X^2 + \cdots = \sum_{i=0}^\infty r_i X^i \in R\llbracket X \rrbracket.
\]
Since there is no ring in $\mathcal{F}$ that contains every $r_i$, we have $f \notin S$, and so $S \subsetneq R \llbracket X \rrbracket$.
\end{rem}


\section{Flatness}

The goal of this section is to show that the ring of Artinian power series has flatness properties that are lacking in the formal power series ring when the base ring is non-Noetherian. Recall that the formal power series ring $R \llbracket X \rrbracket$ is flat over $R$ if $R$ is coherent, but may not be flat over $R$ otherwise. A criterion for flatness of a module over a ring based on linear equations will be useful in investigating the flatness of $\mathcal{F}\llbracket X_1, \ldots, X_n \rrbracket$ over $R$.

\begin{lemma}\label{flatnesscriterion} \cite[Theorem 7.6]{matsumura}
	Let $R$ be a ring, and $M$ an $R$-module. $M$ is flat over $R$ if and only if for all choices $r_1, \ldots, r_n \in R$ and $m_1, \ldots, m_n \in M$ satisfying $\sum_{i=1}^n r_i m_i = 0$, there exist $b_{ij} \in R$ and $y_j \in M$ (for $i = 1 \dots, n$ and $j = 1, \ldots, s$) such that
	\[
		\sum_{i=1}^n r_i b_{ij} = 0 \text{ for each } j, \quad \text{ and } \quad
		m_i = \sum_{j=1}^s b_{ij} y_j \text{ for each }i.
	\]
\end{lemma}

We apply Lemma \ref{flatnesscriterion} to a directed union of rings.

\begin{lemma}\label{directedunionflat}
	Let $\{R_\alpha\}$ be a directed family of rings, and let $\{M_\alpha\}$ be a directed family such that $M_\alpha$ is a flat $R_\alpha$-module for each $\alpha$. Then the directed union $\bigcup_\alpha M_\alpha$ is flat over the directed union $\bigcup_\alpha R_\alpha$.
\end{lemma}
\begin{proof}
	Put $M = \bigcup_\alpha M_\alpha$ and $R = \bigcup_\alpha R_\alpha$. Let $\sum_{i=1}^n r_i m_i = 0$ for $r_i \in R$, $m_i \in M$. Then there exists $\beta$ such that $r_1, \ldots, r_n \in R_{\beta}$ and $m_1, \ldots, m_n \in M_{\beta}$. Since $M_\beta$ is flat over $R_\beta$, by Lemma \ref{flatnesscriterion}, there exist $b_{ij} \in R_\beta \subseteq R$ and $y_j \in M_\beta \subseteq M$ ($j = 1, \ldots, s$) such that $\sum_{i=1}^n r_i b_{ij} = 0$ for each $j$ and $\sum_{j=1}^s b_{ij} y_j = m_i$ for each $i$. 

Hence,  from $\sum_{i=1}^n r_i m_i = 0$ where $r_i \in R$ and $m_i \in M$, we have obtained $b_{ij} \in R$ and $y_j \in M$ such that $\sum_{i=1}^n r_i b_{ij} = 0$ for each $j$ and $\sum_{j=1}^s b_{ij} y_j = m_i$ for each $i$. Thus, by Lemma \ref{flatnesscriterion}, $M$ is flat over $R$, i.e. $\bigcup_\alpha M_\alpha$ is flat over $\bigcup_\alpha R_\alpha$.
\end{proof}

Recall that an $R$-module $M$ is faithfully flat over $R$ if and only if $M$ is flat over $R$ and $M \neq \mathfrak{m}M$ for every maximal ideal $\mathfrak{m}$ of $R$ \cite[Theorem 7.2]{matsumura}. Let $R \subseteq S$ be a ring extension, and let $P$ be a prime ideal of $R$. Then $R \subseteq S$ is said to satisfy the \emph{lying over (LO)} property if there exists a prime ideal $Q$ of $S$ such that $Q \cap R = P$. In this case, the prime ideal $Q$ of $S$ is said to lie over $P$. We record the fact that LO holds automatically for ring extensions over a 0-dimensional ring.

\begin{lemma}\label{LOzero} \cite[Result 2.1]{GH1993}
	Let $R$ be a 0-dimensional ring, and let $R \subseteq S$ be a ring extension. Then $R \subseteq S$ satisfies LO. That is, given any prime ideal $P$ of $R$, there exists a prime ideal $Q$ of $S$ such that $Q \cap R = P$.
\end{lemma}

\begin{thm} \label{ffgeneral}
	Let $R$ be a 0-dimensional ring and let $R_0$ be an Artinian subring of $R$. Write $\mathcal{F} = \mathcal{F}(R_0, R) = \{R_\alpha\}$. Then  the ring $\mathcal{F} \llbracket X_1, \ldots, X_n \rrbracket$ is faithfully flat over $R$.
\end{thm}
\begin{proof}
	Put $S = \mathcal{F} \llbracket X_1, \ldots, X_n \rrbracket =  \bigcup_\alpha (R_\alpha \llbracket X_1, \ldots, X_n \rrbracket)$. For each $R_\alpha \in \mathcal{F}$, let $S_\alpha$ be the formal power series ring $R_\alpha \llbracket X_1, \ldots, X_n \rrbracket$, so we have $S_\alpha = R_\alpha \llbracket X_1, \ldots, X_n \rrbracket$ and $S = \bigcup_\alpha S_\alpha$. Note that $S_\alpha$ is the $(X_1, \ldots, X_n)$-adic completion of $R_\alpha[X_1, \ldots, X_n]$, where $R_\alpha$ is Artinian (hence Noetherian), so $S_\alpha$ is flat over $R_\alpha$ for each $\alpha$. By Lemma \ref{directedunionflat}, $S = \bigcup_\alpha S_\alpha$ is flat over $R = \bigcup_\alpha R_\alpha$.

	Next, by Lemma \ref{LOzero}, $R \subseteq S$ satisfies LO, so for any prime ideal $P \subseteq R$, there is a prime ideal $Q \subseteq S$ such that $Q \cap R = P$. Since $P \subseteq Q$,  we have $PS \subseteq QS = Q \neq S$, so $PS \neq S$, i.e. $P$ survives in $S$. Thus, $S$ is faithfully flat over $R$.
\end{proof}

We have thus obtained a subring $\mathcal{F}\llbracket X_1, \ldots, X_n \rrbracket$ of the power series ring $R\llbracket X_1, \ldots, X_n \rrbracket$ over a 0-dimensional ring $R$ which properly contains the polynomial ring and is faithfully flat over $R$, regardless of whether or not $R$ is coherent. 


\section{Dimension}

We next consider the dimension of $\mathcal{F}\llbracket X_1, \ldots, X_n \rrbracket$. We observe first that there is an upper bound on the dimension, due to our directed union construction.

\begin{lemma}\cite[Remark before Proposition 5.1]{GH1992} \label{upperbound}
	Let $S$ be a ring that is the directed union of a family $\{S_\alpha\}$ of subrings, each of dimension at most $n$. Then $\dim S \leq n$.
\end{lemma}

To obtain a lower bound on the dimension, we consider the map from power series rings over a 0-dimensional ring $R$ to power series rings over a residue field of $R$.

\begin{lemma}\label{a}
	Let $R$ be a 0-dimensional ring and let $M \in \spec R$. Suppose $R$ can be written as the union of a directed family $\{R_\alpha\}$ of Artinian subrings. Consider the map
	\[
		\phi: \bigcup_\alpha \big( R_\alpha \llbracket X_1, \ldots, X_n \rrbracket \big) \longrightarrow (R/M) \llbracket X_1, \ldots, X_n \rrbracket
	\]
	induced by the map on the coefficients of $f \in R_\alpha\llbracket X_1, \ldots, X_n \rrbracket$ given by
	\[
		\phi_{\alpha}: R_\alpha \stackrel{\pi_\alpha}{\longrightarrow} R_\alpha/M_\alpha \cong 	\frac{R_\alpha + M}{M} \subseteq R/M,
	\]
	where $M_\alpha = M \cap R_\alpha$. Then the map $\phi$ is a ring homomorphism with kernel and image
	\[
		\ker \phi = \bigcup_\alpha \big(M_\alpha \llbracket X_1, \ldots, X_n \rrbracket \big) 
		\quad \text{ and } \quad
		\im \phi =  \bigcup_\alpha \big(k_\alpha \llbracket X_1, \ldots, X_n \rrbracket \big),
	\] 
	where $k_\alpha$ denotes the residue field $\displaystyle{R_\alpha / M_\alpha \cong \frac{R_\alpha + M}{M}}$.
\end{lemma}
\begin{proof}
	That $\phi$ is a ring homomorphism is clear since $\phi_\alpha$ is a homomorphism for each $\alpha$. 

	Consider $\ker \phi$ next. Let $f \in \ker \phi$. We have $f \in R_\alpha \llbracket X_1, \ldots, X_n \rrbracket$ for some $\alpha$, and $\phi(f) = 0$. Then every coefficient of $f$ must be in $M_\alpha$, so that $f \in M_\alpha \llbracket X_1, \ldots, X_n \rrbracket$. On the other hand, if $g \in M_\alpha \llbracket X_1, \ldots, X_n \rrbracket$, then $\pi_\alpha: R_\alpha \rightarrow R_\alpha / M_\alpha$ maps every coeffcient to 0, so that $g \in \ker \phi$.

	Lastly, we verify the claim about the image of $\phi$. If $f \in \bigcup_\alpha (R_\alpha \llbracket X_1, \ldots, X_n \rrbracket)$, then $f \in R_\alpha \llbracket X_1, \ldots, X_n \rrbracket$ for some $\alpha$, so $\phi(f) \in k_\alpha \llbracket X_1, \ldots, X_n \rrbracket$. Thus, $\im \phi \subseteq \bigcup_\alpha (k_\alpha \llbracket X_1, \ldots, X_n \rrbracket)$. On the other hand, since $\phi_\alpha: R_\alpha \rightarrow k_\alpha$ is an onto map, for any $\overline{f} \in k_\alpha \llbracket X_1, \ldots, X_n \rrbracket$, we can produce $f \in R_\alpha \llbracket X_1, \ldots, X_n \rrbracket$ with appropriate coefficients so that $\phi(f) = \overline{f}$. Hence,  $\im \phi =  \bigcup_\alpha (k_\alpha \llbracket X_1, \ldots, X_n \rrbracket )$.
\end{proof}

Next, we record the observation that the kernel of the map considered in Lemma \ref{a} is an extension of a maximal ideal of $R$.

\begin{prop}\label{minimalprimes}
	Let $R$ be a 0-dimensional ring that can be written as the union of a directed family $\{R_\alpha\}$ of Artinian subrings, and let $M \in \spec R$. Put $S = \bigcup_\alpha (R_\alpha  \llbracket X_1, \ldots, X_n \rrbracket )$, 
	and for each $\alpha$, let $M_\alpha = M \cap R_\alpha$. Then the extension $MS$ of $M$ in $S$ is given by
	\[
		MS = \bigcup_\alpha (M_\alpha \llbracket X_1, \ldots, X_n \rrbracket ).
	\]
\end{prop}
\begin{proof}
	Let $f \in MS$, and write $f = m_1 f_1 + \cdots + m_s f_s$, where $m_1, \ldots, m_s \in M$ and $f_1, \ldots, f_s \in S$. Since $R$ and $S$ are directed unions, $m_1, \ldots, m_s \in M_\alpha = M \cap R_\alpha$ for some $\alpha$, and $f_1, \ldots, f_s \in R_\beta \llbracket X_1, \ldots, X_n \rrbracket$ for some $\beta$. Then there exists $\gamma$ such that $m_1 f_1 + \cdots + m_s f_s \in M_\gamma \llbracket X_1, \ldots, X_n \rrbracket$, so we have $f \in \bigcup_\alpha (M_\alpha \llbracket X_1, \ldots, X_n \rrbracket)$.

	Next, let $g \in M_\alpha \llbracket X_1, \ldots, X_n \rrbracket$ for some $\alpha$. The coefficients of $g$ come from $M_\alpha$, which is finitely generated, so if $M_\alpha$ is generated by $m_1, \ldots, m_k$, then we can write $g = m_1 g_1 + \cdots + m_k g_k$ for some $g_1, \ldots, g_k \in R_\alpha \llbracket X_1, \ldots, X_n \rrbracket$. Hence,
	\[
		g \in M_\alpha (R_\alpha \llbracket X_1, \ldots, X_n \rrbracket) \subseteq M(R_\alpha \llbracket X_1, \ldots, X_n \rrbracket) \subseteq M \Big( \bigcup_\alpha \big( R_\alpha \llbracket X_1, \ldots, X_n \rrbracket \big) \Big) = M S.
	\] This establishes $MS = \bigcup_\alpha (M_\alpha \llbracket X_1, \ldots, X_n \rrbracket )$.
%
\end{proof}

The following result due to Magarian \cite{magarian} will be useful in our setting.

\begin{lemma}\label{tensorpower2} \cite[Theorem 3]{magarian}
	Let $L$ be an algebraic extension field of a field $k$. Let $R = k \llbracket X_1, \ldots, X_n \rrbracket$. Then $R \otimes_k L$ is a regular local ring of dimension $n$.
\end{lemma}

We can apply this lemma to a 0-dimensional pair $(R_0, R)$. We note that for $M \in \spec R$, the residue field $R/M$ is an algebraic field extension of $R_0 / (M \cap R_0)$. So $R/M$ is the directed union of the family of all finite algebraic extensions of $R_0 / (M \cap R_0)$. 

\begin{rem} \label{tensorequations}
Let $R$ be a 0-dimensional ring that is integral over an Artinian subring $R_0$, and write $\mathcal{F} = \mathcal{F}(R_0, R) = \{ R_\alpha\}$. If $M \in \spec R$, the residue field $R/M$ is a directed union of the family of fields $\{k_\alpha\}$, where $k_\alpha = R_\alpha / (M \cap R_\alpha) \cong (R_\alpha + M)/M$, and each $k_\alpha$ is finitely generated over $k_0 = R_0 / (M \cap R_0) \cong (R_0 + M)/M$ since $R_\alpha$ is finitely generated over $R_0$. Then, using the fact  that the direct limit commutes with the tensor product, we obtain
\begin{align*}
	k_0 \llbracket X_1, \ldots, X_n \rrbracket  \otimes_{k_0} (R/M) 
	 &\cong k_0 \llbracket X_1, \ldots, X_n \rrbracket \otimes_{k_0} \bigcup_\alpha (R_\alpha + M)/M\\
	 &\cong  k_0 \llbracket X_1, \ldots, X_n \rrbracket \otimes_{k_0} \varinjlim (R_\alpha + M)/M \\
	 &\cong  k_0 \llbracket X_1, \ldots, X_n \rrbracket \otimes_{k_0} \varinjlim k_\alpha \\
	 &\cong \varinjlim \big( k_0 \llbracket X_1, \ldots, X_n \rrbracket \otimes_{k_0} k_\alpha) \\
	 &\cong \varinjlim (k_\alpha \llbracket X_1, \ldots, X_n \rrbracket )
	 = \bigcup_\alpha (k_\alpha \llbracket X_1, \ldots, X_n \rrbracket).
\end{align*}
\end{rem}

Combining Remark \ref{tensorequations} with Lemma \ref{tensorpower2}, we get the following dimension result.

\begin{thm} \label{regular}
	Let $R$ be a 0-dimensional ring that is integral over an Artinian ring $R_0$. Let $\mathcal{F} = \mathcal{F}(R_0, R) = \{R_\alpha\}$. Then
	\[
		\dim  \mathcal{F}\llbracket X_1, \ldots, X_n \rrbracket = n.
	\]
	For each $M \in \spec R$, $M \mathcal{F}\llbracket X_1, \ldots, X_n \rrbracket$ is a minimal prime ideal of $\mathcal{F}\llbracket X_1, \ldots, X_n \rrbracket$, and the quotient of $\mathcal{F}\llbracket X_1, \ldots, X_n \rrbracket$ by $M \mathcal{F}\llbracket X_1, \ldots, X_n \rrbracket$ is a regular local ring of dimension $n$. 
\end{thm}
\begin{proof}
	Put $S = \mathcal{F}\llbracket X_1, \ldots, X_n \rrbracket$. By Lemma \ref{a}, for $k_\alpha \cong (R_\alpha + M) / M$, we have
	\[
		S / MS  \, \cong\,  \bigcup_\alpha (k_\alpha \llbracket X_1, \ldots, X_n \rrbracket),
	\]
	and by Remark \ref{tensorequations}, we have
	\[
		\bigcup_\alpha \big(k_\alpha \llbracket X_1, \ldots, X_n \rrbracket\big) \cong k_0 \llbracket X_1, \ldots, X_n \rrbracket \otimes_{k_0} (R/M).
	\]
	By Lemma \ref{tensorpower2}, this ring is a regular local ring of dimension $n$, which implies $S$ has dimension at least $n$ by Lemma \ref{a}, i.e. $\dim S \geq n$. We also have $\dim S \leq n$ by Lemma \ref{upperbound}. Combining the two inequalities, we obtain $\dim S = n$. Since $S$ has a homomorphic image that is a regular local ring of dimension $n$, hence a domain \cite[Theorem 164]{kaplansky}, with kernel the extension of $M$ in $S$, we see that $MS$ is a minimal prime ideal.
\end{proof}

We can now summarize flatness and dimension properties in the following table.

\newcolumntype{u}{>{\columncolor[RGB]{225,225,255}} m{2.1cm}}
\newcolumntype{s}{>{\columncolor[RGB]{225,225,255}} m{4.0cm}}
\newcolumntype{v}{>{\columncolor[RGB]{255,225,235}} m{4.2cm}}
 \newcolumntype{w}{>{\columncolor[RGB]{255,200,200}} m{2.8cm}}
  \newcolumntype{q}{m{2.0cm}}
\newcolumntype{x}{>{\columncolor[RGB]{255,225,235}} m{5.0cm}}

\begin{table}[!h]
\begin{center}
\begin{tabular}{ q| s| v | w | } 

\cline{2-4}
 & \multicolumn{1}{c|}{\textbf{Noetherian Ring} $R$}
           & \multicolumn{2}{c|}{\textbf{Non-Noetherian Ring} $R$} \\
       & \centering $ R \llbracket X_1, \ldots, X_n \rrbracket$ 
         &  \centering $ R \llbracket X_1, \ldots, X_n \rrbracket $ 
         & \centering $\mathcal{F}\llbracket X_1, \ldots, X_n \rrbracket$  \tabularnewline
      \cline{2-4}
        \centering  {  Dimension } &  \centering $n$ \cellcolor[RGB]{240,240,255}  &\quad \ \centering $n$ if $R$ is \emph{SFT} \newline {\small ($\infty$ otherwise) }\cellcolor[RGB]{255,240,248}  &  \centering $n$ \cellcolor[RGB]{255,240,240}  \tabularnewline \hline
        \centering  {  Flatness } & \centering $\quad$ Faithfully Flat $\quad$ over $R$ &  \centering  { Faithfully Flat over $R$} \newline  \hspace{0.5in} { if $R$ is \emph{coherent}} \hspace{-0.28in} \vspace{-0.03in} \newline  { \scriptsize (may not be flat otherwise) }  &  \centering Faithfully Flat over $R$  \tabularnewline
\end{tabular}
\end{center}
\caption{Dimension and flatness of power series rings over a 0-dimensional ring $R$ that is integral over an Artinian subring $R_0$.}
\end{table}


\section{Prime Spectrum and Heights of Finitely Generated Ideals}

Theorem \ref{regular} describes some of the prime ideals of $\mathcal{F}\llbracket X_1, \ldots, X_n \rrbracket$. We continue investigating the prime spectrum of $\mathcal{F}\llbracket X_1, \ldots, X_n \rrbracket$ in this section. This will lead to a result about the relationship between the height of a finitely generated ideal and the number of its generators. We begin with a complete characterization of the minimal prime ideals of $\mathcal{F}\llbracket X_1, \ldots, X_n \rrbracket$.

\begin{lemma}\label{minlocal}
	Let $R$ be a 0-dimensional ring that is integral over an Artinian ring $R_0$. Write $\mathcal{F} = \mathcal{F}(R_0, R) =  \{R_\alpha\}$. Then the set of minimal prime ideals of $S = \mathcal{F}\llbracket X_1, \ldots, X_n \rrbracket$ consists of extensions of the prime ideals of $R$ to $S$, i.e.
	\[
		\text{Minimal prime ideals of } S =  \{ M S\ | \ M \in \spec R \}.
	\] 
\end{lemma}
\begin{proof}
	By Theorem \ref{regular}, for any $M \in \spec R$, the extension of $M$ is a minimal prime ideal of $S$.

	Next, let $P$ be a minimal prime ideal of $S$. Then $P \cap R \in \spec R$, so $P \cap R = M$ for some $M \in \spec R$. Then $M \subseteq P$, so we have $M S \subseteq P$. By the minimality of $P$, we have $P = M S$.
\end{proof}

%

Recall that a ring is catenary if, for any pair of prime ideals $P \subseteq Q$, all saturated chains of primes ideals between $P$ and $Q$ have the same length.
\begin{prop}\label{catenary}
	Let $R$ be a 0-dimensional ring that is integral over an Artinian ring $R_0$. Let $\mathcal{F} = \mathcal{F}(R_0, R) =  \{R_\alpha\}$. Then $\mathcal{F} \llbracket X_1, \ldots, X_n \rrbracket$ is catenary. 
\end{prop}
\begin{proof}
	Put $S = \mathcal{F} \llbracket X_1, \ldots, X_n \rrbracket$. Let $P \subseteq Q$ be prime ideals in $S$. Let $P_0$ be a minimal prime ideal of $S$ contained in $P$. By Proposition \ref{minlocal}, we can write $P_0 = MS$ for some $M \in \spec R$. Consider the quotient $S / MS = \overline{S}$, which is a regular local ring by Theorem \ref{regular}, so $\overline{S}$ is catenary \cite[Theorem 17.4 (ii)]{matsumura}. Hence, every saturated chain of prime ideals between the images of $P$ and $Q$ in $\overline{S}$ has the same length, and so this is also true between $P$ and $Q$ inside $S$, meaning $S$ is catenary.
\end{proof}

Every Noetherian ring satisfies Krull's Height Theorem, i.e. an ideal that is generated by $t$ elements has height at most $t$. We show next that Krull's Height Theorem holds for finitely generated ideals in $\mathcal{F}\llbracket X_1, \ldots, X_n \rrbracket$. It will be convenient to introduce the notion of height-generated ideals.

\begin{de}
	A finitely generated ideal $I$ of a ring is said to be \emph{height-generated} if it can be generated by $\Ht(I)$-many elements.
\end{de}
 
\begin{prop} \label{krull}
	Let $R$ be a 0-dimensional ring that is integral over an Artinian ring $R_0$, and write $\mathcal{F} = \mathcal{F}(R_0, R) =  \{R_\alpha\}$. Let $(f_1, \ldots, f_k)$ be a finitely generated ideal in $\mathcal{F}\llbracket X_1, \ldots, X_n \rrbracket$. Then $\text{ht}(f_1, \ldots, f_k) \leq k$. 
\end{prop}
\begin{proof}
	Put $S = \mathcal{F}\llbracket X_1, \ldots, X_n \rrbracket$, and let $M \in \spec R$. The image of the ideal 
	\[
		(f_1, \ldots, f_k, M) = (f_1, \ldots, f_k)S + M S
	\]
	in the quotient $\overline{S} = S / MS$ 
	is generated by the images of $f_1, \ldots, f_k$.  Since $\overline{S}$ is a regular local ring by Theorem \ref{regular}, Krull's Height Theorem holds in $\overline{S}$, so that $\Ht(\overline{f_1}, \ldots, \overline{f_k}) \leq k$, where $\overline{f_i}$ is the image of $f_i$ in $\overline{S}$. Hence, there is a prime ideal $P \subseteq S$ containing $MS$ such that $(\overline{f_1}, \ldots, \overline{f_k}) \subseteq \overline{P}$ and $\Ht(\overline{P}) \leq k$ in $\overline{S}$. Since $MS$ is minimal in $S$ and $MS \subseteq P$, we also have $\Ht(P) \leq k$ in $S$. Moreover, $(f_1, \ldots, f_k) \subseteq P$ in $S$, so we obtain $\text{ht}(f_1, \ldots, f_k) \leq k$, as we wanted to show.
\end{proof}


\begin{cor}
	Let $R$ be a 0-dimensional ring that is integral over an Artinian ring $R_0$. Let $\mathcal{F} = \mathcal{F}(R_0, R) =  \{R_\alpha\}$. Let $J$ be a height-generated ideal of $\mathcal{F} \llbracket X_1, \ldots, X_n \rrbracket$, and let $I/J$ be an ideal of $\mathcal{F}\llbracket X_1, \ldots, X_n \rrbracket / J$ generated by $t$ elements. Then the height of $I/J$ in the ring $\mathcal{F} \llbracket X_1, \ldots, X_n \rrbracket/J$ is at most $t$.
\end{cor}
\begin{proof}
	Put $S = \mathcal{F} \llbracket X_1, \ldots, X_n \rrbracket$. Let $P/J$ be a minimal prime divisor of $I/J$. We show that $\Ht(P/J) \leq t$. Put $k = \Ht_{S}(J)$, and since $J$ is a height-generated ideal of $S$, $J$ can be generated by $k$ elements. This means the ideal $I \subseteq S$ can be generated by $k + t$ elements.

	Since $P/J$ is a minimal prime divisor of $I/J$ in $S/J$, $P$ is a minimal prime divisor of $I$ in $S$. And since $J \subseteq I \subseteq P$ in $S$, we see that $P$ contains a minimal prime divisor $Q$ of $J$, where $\Ht_{S}(Q) = k$ since $J$ is height-generated. Since $I$ can be generated by $k + t$ elements, by Proposition \ref{krull} applied to $I \subseteq S$, we have $\Ht_{S}(P) \leq k + t$. Since $S$ is catenary by Proposition \ref{catenary}, we have $\Ht(P/Q) = \Ht_{S}(P) - \Ht_{S}(Q) \leq k + t - k = t$. Then $\Ht(P/J) \leq \Ht(P/Q) + \Ht(Q/J) \leq t + 0 = t$. So $\Ht(P/J) \leq t$, which is what we wanted to show.
\end{proof}



\begin{cor} \label{regseqheightgenart}
	Let $R$ be a 0-dimensional ring that is integral over an Artinian ring $R_0$. Let $\mathcal{F} = \mathcal{F}(R_0, R) =  \{R_\alpha\}$. Let $f_1, \ldots, f_t$ be a regular sequence in $\mathcal{F}\llbracket X_1, \ldots, X_n \rrbracket$. Then for $I = (f_1, \ldots, f_t)$, we have $\Ht(I) = t$, so $I$ is a height-generated ideal.
\end{cor}
\begin{proof}
	By Proposition \ref{krull}, we have $\Ht(I) \leq t$. Since $I$ is generated by a regular sequence, we also have $\Ht(I) \geq t$ \cite[Theorem 132]{kaplansky}. This means we have $\Ht(I) = \Ht(f_1, \ldots, f_t) = t$, and so $I$ is a height-generated ideal.
\end{proof}

We show next that if $R$ is semilocal, then the prime spectrum of $\mathcal{F}\llbracket X_1, \ldots, X_n \rrbracket$ is Noetherian. Recall that any 0-dimensional semilocal ring $R$ is integral over some Artinian subring $R_0$, by Proposition \ref{gilmerheinzer}. 

\begin{prop} \label{noethspectrum}
	Let $R$ be a 0-dimensional semilocal ring, integral over an Artinian ring $R_0$. Let $\mathcal{F} = \mathcal{F}(R_0, R) =  \{R_\alpha\}$. Then $\mathcal{F} \llbracket X_1, \ldots, X_n \rrbracket$ has Noetherian prime spectrum.
\end{prop}
\begin{proof}
	Put $S = \mathcal{F} \llbracket X_1, \ldots, X_n \rrbracket$. Since $R$ is semilocal, $\spec R = \{M_1, \ldots, M_k\}$ is finite. By Lemma \ref{minlocal}, the set of minimal prime ideals of $S$ is
	\[
		 \big\{ M_j \cdot S 
		\ \big| \ j = 1, \ldots, k \big\}.
	\]
	For $s \neq t$, since $M_s$ and $M_t$ are maximal in $R$, we have $1 \in M_s + M_t$. Then the extensions of $M_s$ and $M_t$ in $S$ are coprime. We then have
	\[
		S / \sqrt{(0)} = S / (M_1 S) \cdots (M_k S) \cong \big( S / M_1 S\big)  \times \cdots \times \big( S/ M_k S \big).
	\]
	Now, $\spec S \approx \spec (S/\sqrt{(0)})$, and for a finite direct product of rings $A_1 \times \cdots \times A_k$, we have 
	\[
		\spec (A_1 \times \cdots \times A_k) = \{ A_1 \times \cdots \times P_i \times \cdots \times A_k \ | \ P_i \in \spec (A_i) \}.
	\]
	Since the prime spectrum of a finite direct product of rings is a finite disjoint union of the prime spectra, we obtain
	\[
		\spec S \approx \spec  \big( S / M_1 S\big)  \sqcup \cdots \sqcup \spec \big( S/ M_k S \big).
	\]
	Each $S/M_j S$ is a (Noetherian) regular local ring by Theorem \ref{regular}, so each $S/M_j S$ has Noetherian prime spectrum. As a finite disjoint union of Noetherian prime spectra, $S = \mathcal{F} \llbracket X_1, \ldots, X_n \rrbracket$ also has Noetherian prime spectrum.
\end{proof}


\section{Non-Noetherian Cohen-Macaulayness}

In this section, we examine properties of $\mathcal{F}\llbracket X_1, \ldots, X_n \rrbracket$ that resemble those seen in Cohen-Macaulay rings. A defining feature of Noetherian Cohen-Macaulay rings is that the length of a maximal regular sequence in a proper ideal $I$ equals the height of $I$, i.e. the ``{grade = height}" property. Following \cite{at}, let the {classical grade} of an ideal $I$ denote the supremum of the lengths of regular sequences in $I$, i.e. 
\[
	\cgrade_R(I) = \sup\{ i \, | \, f_1, \ldots, f_i \text{ is a regular sequence in } I \}.
\]
In general, we have $\cgrade_R(I) \leq \Ht(I)$. If $R$ is a Cohen-Macaulay ring, then $\cgrade_R(I) = \Ht(I)$.

For a Noetherian ring $R$, every maximal regular sequence in an ideal $I$ has the same length \cite[Theorem 121]{kaplansky}. The classical grade in a Noetherian ring also has a homological characterization \cite[Theorem 16.7]{matsumura}, given by
\[
		\cgrade_R(I) =  \inf\{ \, i \,  |  \, \Ext_R^i(R/I, R) \neq 0 \, \}.
\]

Cohen-Macaulay rings can also be characterized by looking at the set of associated primes of a height-generated ideal. If $R$ is a Noetherian ring and $I$ is an ideal of $R$, recall that a prime ideal $P$ that can be written as $(I:_R a)$ for some $a \in R \setminus I$ is called an \emph{associated prime} of $I$. The set of associated primes of $I$ is finite \cite[Theorem 6.5(i)]{matsumura}, and the set of zero-divisors for the $R$-module $R/I$ is the union of all the associated primes of $I$ \cite[Theorem 6.1(i)]{matsumura}.

In Cohen-Macaulay rings, all associated primes of a height-generated ideal $I$ have the same height as $I$. This is the so-called {unmixedness} property of Cohen-Macaulay rings. Stated another way, the height-generated ideals of Cohen-Macaulay rings have no embedded primes, where an embedded prime is an associated prime of $I$ that is not minimal over $I$.


\subsection{Unmixedness}

We begin by giving a definition for what it means for a non-Noetherian ring to satisfy the unmixedness property, then we consider unmixedness for $\mathcal{F}\llbracket X_1, \ldots, X_n \rrbracket$ for the case that $R$ is semilocal. We begin by defining generalized notions of prime divisors and associated primes, following the exposition in \cite{fuchsheinzerolberding}.
 
Let $R$ be a ring and $I$ an ideal of $R$. If $a \in R$ is such that $I \subsetneq (I:_R a)$, we say that $a$ is \emph{not prime} to $I$. The set
\[
	S(I) = \bigcup_{a \in R \setminus I} (I:_R a) = \{ x \in R \ | \ ax \in I \text{ for some } a \in R \setminus I\}
\]
	is the set of elements that are not prime to $I$. Observe that for $x \in R$, we have $x \notin S(I)$ if and only if $x$ is $(R/I)$-regular. A prime ideal $P$ is a \emph{prime divisor} of $I$ if $P \subseteq S(I)$. A prime divisor $P$ of $I$ is called a \emph{maximal} (resp. \emph{minimal}) prime divisor of $I$ if $P$ is maximal (resp. minimal) in $S(I)$ with respect to inclusion. If $R$ is a non-Noetherian ring, it is possible that a prime divisor $P$ of $I$ cannot be written in the form $(I:_R a)$ for some $a \in R \setminus I$.

\begin{de} \cite[Section 1]{fuchsheinzerolberding} \label{nnassoc}
	Let $R$ be a ring, and let $P$ be a prime divisor of an ideal $I$. 
	\begin{enumerate}
		\item $P$ is a \emph{weak-Bourbaki associated} (or \emph{weakly associated}) \emph{prime} of $I$ if $P$ is minimal over $(I:_R a)$ for some $a \in R \setminus I$.
		\item $P$ is a \emph{Zariski-Samuel associated prime} of $I$ if $P = \sqrt{(I:_R a)}$ for some $a \in R \setminus I$.
		\item $P$ is a \emph{Krull associated prime} of $I$ if $P$ is a union of ideals of the form $(I :_R a)$.
	\end{enumerate} 
\end{de}

Note that for a proper ideal $I$ of a ring $R$, $S(I) = I$ if and only if $I$ is a prime ideal of $R$. Moreover, the complement of $S(I)$ is a saturated multiplicatively closed set, so $S(I)$ is a set-theoretic union of prime ideals \cite[p.34]{kaplansky}.


\begin{de} \label{defunmixed}
	Let $R$ be a ring. An ideal $I$ is \emph{unmixed} if all prime divisors of $I$ have the same height. The ring $R$ satisfies the \emph{unmixedness property} if every height-generated ideal of $R$ is unmixed.
\end{de}

To prove statements about unmixedness, it will be useful to know when extensions of power series rings are integral. The next lemma gives a condition for integrality in power series ring extensions.


\begin{lemma} \label{integralbasering}
	Let $R_\alpha \subseteq R_\beta$ be an extension of rings such that $R_\beta$ is finitely generated as an $R_\alpha$-module. Then $R_\alpha \llbracket X_1, \ldots, X_n \rrbracket \subseteq R_\beta \llbracket X_1, \ldots, X_n \rrbracket$ is an integral extension. In particular, if $k_\alpha \subseteq k_\beta$ is a finite field extension, then $k_\alpha \llbracket X_1, \ldots, X_n \rrbracket \subseteq k_\beta \llbracket X_1, \ldots, X_n \rrbracket$ is an integral extension.
\end{lemma}
\begin{proof}
	Write $S_\alpha = R_\alpha \llbracket X_1, \ldots, X_n \rrbracket$ and $S_\beta = R_\beta \llbracket X_1, \ldots, X_n \rrbracket$. Let $f \in S_\beta$. Write $f = \sum_\delta r_\delta X^\delta$ with $r_\delta \in R_\beta$, where $\delta = (\delta_1, \ldots, \delta_n) \in \mathbb{N}^n$, and $X^\delta$ stands for $X_1^{\delta_1} \cdots X_n^{\delta_n}$.

	Since $R_\beta$ is a finitely generated $R_\alpha$-module, we can write $R_\beta = R_\alpha  b_1 + \cdots + R_\alpha  b_k$ for some $b_1, \ldots, b_k \in R_\beta$. Then for each $\delta$, we can write $r_\delta = a_{\delta, 1} b_1 + \cdots + a_{\delta, k} b_k$ for some $a_{\delta,1}, \ldots, a_{\delta,k} \in R_\alpha$. Then we have
	\begin{align*}
		f = \sum_\delta r_\delta X^\delta &= \sum_\delta (a_{\delta, 1} b_1 + \cdots + a_{\delta, k} b_k) X^\delta \\
		&= b_1 \sum_\delta  a_{\delta, 1} X^\delta + \cdots + b_k \sum_\delta a_{\delta, k} X^\delta
		\in  S_\alpha[b_1, \ldots, b_k] \subseteq S_\beta
	\end{align*}
	where $S_\alpha [b_1, \ldots, b_k]$ is integral over $S_\alpha$ since $b_1, \ldots, b_k$ are integral over $R_\alpha$ \cite[Theorem 9.1]{matsumura}. Hence, $f$ is integral over $S_\alpha$. The choice $f \in S_\beta$ was arbitrary, so $S_\alpha \subseteq S_\beta$ is an integral extension.

\end{proof}

\begin{prop}\label{integralext}
	Let $R$ be a 0-dimensional ring that is integral over an Artinian subring $R_0$. Write $\mathcal{F} = \mathcal{F}(R_0, R) = \{R_\alpha\}$. Then $\mathcal{F} \llbracket X_1, \ldots, X_n \rrbracket$ is an integral extension of $R_\alpha \llbracket X_1, \ldots, X_n \rrbracket$ for each $\alpha$.
\end{prop}
\begin{proof}
	Let $S = \mathcal{F} \llbracket X_1, \ldots, X_n \rrbracket$, and let $f \in S$. Then $f \in R_\alpha \llbracket X_1, \ldots, X_n \rrbracket$ for some $\alpha$ where $R_\alpha$ is finitely generated as an $R_0$-module. By Lemma \ref{integralbasering}, $R_0 \llbracket X_1, \ldots, X_n \rrbracket \subseteq R_\alpha \llbracket X_1, \ldots, X_n \rrbracket$ is an integral extension, so $f$ is integral over $R_0 \llbracket X_1, \ldots, X_n \rrbracket$. Our initial choice of $f \in S$ was arbitrary, so $S$ is integral over $R_0 \llbracket X_1, \ldots, X_n \rrbracket$. For each $\alpha$, we have the chain of ring extensions
	\[
		R_0 \llbracket X_1, \ldots, X_n \rrbracket \subseteq R_\alpha \llbracket X_1, \ldots, X_n \rrbracket \subseteq S,
	\]
	so $S$ is also integral over $R_\alpha \llbracket X_1, \ldots, X_n \rrbracket$ for each $\alpha$.
\end{proof}

Let $R$ be a semilocal 0-dimensional ring, and let $R_0$ be an Artinian subring of $R$ that is large in $R$. For the directed family $\mathcal{F} = \mathcal{F}(R_0, R) = \{R_\alpha\}$ of Artinian subrings of $R$, observe that $R_\alpha \llbracket X_1, \ldots, X_n \rrbracket$ is Cohen-Macaulay for each $\alpha$ since it is a power series ring over an Artinian ring. Hence, $\mathcal{F}\llbracket X_1, \ldots, X_n \rrbracket$ is a directed union of Cohen-Macaulay subrings that is an integral extension of $R_\alpha \llbracket X_1, \ldots, X_n \rrbracket$ for each $\alpha$. Moreover, $\mathcal{F} \llbracket X_1, \ldots, X_n \rrbracket$ has Noetherian prime spectrum, by Proposition \ref{noethspectrum}. We apply these observations to unmixedness in $\mathcal{F}\llbracket X_1, \ldots, X_n \rrbracket$.

%
%
%
%


\begin{lemma} \label{unmixedness} \cite[Lemma 4.1]{olb1}
	Let $S = \bigcup_\alpha S_\alpha$ be a directed union of Cohen-Macaulay subrings such that $S_\alpha \subseteq S$ is an integral extension for each $\alpha$. Let $I$ be a proper height-generated ideal of $S$. Then every element of $S$ not prime to $I$ is an element of a minimal prime divisor of $I$ having the same height as $I$. If $S$ also has Noetherian prime spectrum, then every maximal prime divisor of $I$ is a minimal prime divisor of $I$ having the same height as $I$.
\end{lemma}

\begin{thm} \label{unmixed}
	Let $R$ be a 0-dimensional semilocal ring and $R_0$ a large Artinian subring of $R$. Write $\mathcal{F} = \mathcal{F}(R_0, R) =  \{R_\alpha\}$. Let $I$ be a height-generated proper ideal of $\mathcal{F} \llbracket X_1, \ldots, X_n \rrbracket$. Then $I$ has only finitely many maximal prime divisors, each of which is a minimal prime divisor of the same height as $I$, and we have
	\[
		\{ \text{maximal prime divisors of } I \} = \{ \text{minimal prime divisors of }I\}.
	\]
\end{thm}
\begin{proof}
	Let $S = \mathcal{F} \llbracket X_1, \ldots, X_n \rrbracket =  \bigcup_\alpha (R_\alpha \llbracket X_1, \ldots, X_n \rrbracket)$. The ring $S$ has Noetherian prime spectrum by Proposition \ref{noethspectrum}, and $S$ is a directed union of Cohen-Macaulay subrings such that $S$ is integral over $R_\alpha \llbracket X_1, \ldots, X_n \rrbracket$ for each $\alpha$, by Proposition \ref{integralext}. By Lemma \ref{unmixedness}, since $I$ is a height-generated ideal of $S$, the set of maximal prime divisors of $I$ is finite, and it coincides with the set of minimal prime divisors of $I$. From Lemma \ref{unmixedness}, we also see that every prime divisor of $I$ has the same height as $I$.
\end{proof}

Theorem \ref{unmixed} shows that for any height-generated ideal $I$ of $\mathcal{F} \llbracket X_1, \ldots, X_n \rrbracket$, every prime divisor of $I$ is both maximal and minimal, having the same height as $I$, so $I$ is unmixed and $\mathcal{F} \llbracket X_1, \ldots, X_n \rrbracket$ satisfies the unmixedness property in the sense of Definition \ref{defunmixed}. Moreover, all three notions given in Definition \ref{nnassoc} of associated prime ideals over a height-generated ideal coincide in $\mathcal{F} \llbracket X_1, \ldots, X_n \rrbracket$.

\subsection{Regular Sequences}

Next, we examine statements about regular sequences in $\mathcal{F} \llbracket X_1, \ldots, X_n \rrbracket$. 

\begin{thm} \label{gradeheight}
	Let $R$ be a 0-dimensional semilocal ring, and let $R_0$ and $\mathcal{F}$ be as in Theorem \ref{unmixed}. Let $I$ be an ideal of $\mathcal{F} \llbracket X_1, \ldots, X_n \rrbracket$. Then every maximal regular sequence in $I$ has length $\Ht(I)$, i.e. $\cgrade(I) = \Ht(I)$.
\end{thm}
\begin{proof}
	Put $S = \mathcal{F} \llbracket X_1, \ldots, X_n \rrbracket$. Let $t = \Ht_S(I)$, and suppose $f_1, f_2, \ldots, f_k$ is a regular sequence in $I$. Let $J_k =(f_1, \ldots, f_k)S$, and by Corollary \ref{regseqheightgenart}, we have $k = \Ht(J_k) \leq \Ht(I) = t$. If $k = t$, then $f_1, \ldots, f_t$ is a maximal regular sequence in $I$. 
	
	Next, suppose $k < t = \Ht(I)$. 
	 Since $J_k$ is a height-generated ideal, by Theorem \ref{unmixed}, there are finitely many prime divisors of $J_k$, say $Q_1, \ldots, Q_m$, each of height $k$. Since $\Ht(I) > k$, we have $I \not\subseteq Q_1 \cup \cdots \cup Q_m = S(J_k)$ by prime avoidance. We can then find an element $f_{k+1} \in I$ that is $S/J_k$-regular, and so $f_1, \ldots, f_k, f_{k+1}$ is a regular sequence in $S$. We can continue this process to obtain a maximal regular sequence of length equal to $\Ht(I)$. Hence, every maximal regular sequence in $I$ will have length $\Ht(I)$, and we can conclude $\cgrade (I) = \Ht(I)$.
\end{proof}

Theorem \ref{gradeheight} shows that we have the ``grade = height" property in $\mathcal{F}\llbracket X_1, \ldots, X_n \rrbracket$. We proceed to show that height-generated ideals are generated by regular sequences. We will use the following lemma which is a stronger version of prime avoidance.

%

\begin{lemma} \label{regularreplacement} \cite[Theorem 124]{kaplansky}
	Let $P_1, \ldots, P_n$ be prime ideals in a commutative ring $S$, and let $I$ be an ideal of $S$, and $x \in S$ such that $(x, I) \not\subseteq P_1 \cup \cdots \cup P_n$. Then there exists an element $y \in I$ such that $x + y \notin P_1 \cup \cdots \cup P_n$.
\end{lemma}

\begin{thm}
	Let $R$ be a 0-dimensional semilocal ring, and let $R_0$ and $\mathcal{F}$ be as in Theorem \ref{unmixed}. Then the height-generated ideals of $\mathcal{F}\llbracket X_1, \ldots, X_n \rrbracket$ are precisely the ideals generated by a regular sequence.
\end{thm}
\begin{proof}
	If $I$ is generated by a regular sequence of length $t$, then we have $\text{ht}(I) = t$ by Corollary \ref{regseqheightgenart}. So it suffices to show that height-generated ideals can be generated by a regular sequence. 

	Put $S = \mathcal{F}\llbracket X_1, \ldots, X_n \rrbracket$, and let $I = (f_1, \ldots, f_t)S$ with $t = \text{ht}(I)$. If $t = 0$, then $I$ is generated by the empty sequence, and the statement holds. Next, assume $t \geq 1$. For $\spec R = \{M_1, \ldots, M_k \}$, by Lemma \ref{minlocal}, the minimal prime ideals of $S$ are $\{Q_1, \ldots, Q_k\}$, where $Q_i = M_i \cdot S$. Since $\Ht(I) \geq 1$, we have $I \not\subseteq Q_1 \cup \cdots \cup Q_k$, so by Lemma \ref{regularreplacement} applied to $(f_1, J)$ where $J = (f_2, \ldots, f_t)$, there exists an element $y \in J$ such that $f_1 + y \notin Q_1 \cup \cdots \cup Q_k$. Put $h = f_1 + y$, and we have
	\[
		h = f_1 + g_2 f_2 + \cdots + g_t f_t \ \text{ for some }  g_2, \ldots, g_t \in S.
	\] 
	We observe that $I = (f_1, f_2, \ldots, f_t)S = (h, f_2, \ldots, f_t)S$ since $f_1 \in (h, f_2, \ldots, f_t)$. So   if $f_1$ is not a regular element, we can replace $f_1$ with a regular element.

	Next, suppose $f_1, \ldots, f_r$ is a regular sequence for $1 \leq r \leq t$. If $r = t$, then the ideal $I$ is already generated by a regular sequence. If $r < t$, let $I_r = (f_1, \ldots, f_r)$, so $I_r$ is generated by a regular sequence. We have $\Ht(I_r) = r$ by Corollary \ref{regseqheightgenart}, and by Theorem \ref{unmixed}, there are only finitely many prime divisors of $I_r$, say $P_1, \ldots, P_m$, each with height $r$. Since $\Ht(I) = t > r$, we have $I \not\subseteq P_1 \cup \cdots \cup P_m$ by prime avoidance. In particular, one of $f_{r+1}, \ldots, f_t$ is not contained in $P_1 \cup \cdots \cup P_m$. Without loss of generality, assume $f_{r+1} \notin P_1 \cup \cdots \cup P_m$. Then we can apply Lemma \ref{regularreplacement} to $(f_{r+1}, I_r)$ to produce an element $h_{r+1} = f_{r+1} + y_r$ such that $h_{r+1} \notin P_1 \cup \cdots \cup P_m$ and $y_r \in I_r$. We see that $f_1, \ldots, f_r, h_{r+1}$ is a regular sequence since $h_{r+1}$ is $S/I_r$-regular, and we also have $(f_1, \ldots, f_r, h_{r+1}, f_{r+2}, \ldots, f_t) = (f_1, \ldots, f_{r+1}, \ldots, f_t) = I$ since $h_{r+1} \in (f_1, \ldots, f_{r+1})$ and $f_{r+1} \in (f_1, \ldots, f_r, h_{r+1})$. Continuing in this way, we can keep replacing the generating set of $I$ with a regular sequence and increase $r$ incrementally until we reach $r = t$. We can thus conclude that a height-generated ideal $I$ can be generated by a regular sequence.
\end{proof}

\begin{cor}
	Let $R$ be a 0-dimensional semilocal ring, and let $R_0$ and $\mathcal{F}$ be as in Theorem \ref{unmixed}. Put $S = \mathcal{F} \llbracket X_1, \ldots, X_n \rrbracket$. Then $f_1, \ldots, f_t$ is a regular sequence in $S$ if and only if $\Ht((f_1, \ldots, f_i)S) = i$ for all $1 \leq i \leq t$.
\end{cor}
\begin{proof}
	If $f_1, \ldots, f_t$ is a regular sequence in $S$, then $(f_1, \ldots, f_i)S$ has height $i$ for each $1 \leq i \leq t$, by Corollary \ref{regseqheightgenart}. Conversely, suppose $\Ht((f_1, \ldots, f_i)S) = i$ for all $1 \leq i \leq t$. Since $\Ht(f_1 S) = 1$, we see that $f_1$ is not contained in any minimal prime ideal of $S$, and by Theorem \ref{unmixed}, the set of zero-divisors of $S$ is the union of the minimal prime ideals of $S$. This implies $f_1$ is a regular element. Next, suppose $f_1, \ldots, f_i$ is regular for $1 \leq i < t$. Then $I_i = (f_1, \ldots, f_i)$ is a height-generated ideal, so all primes divisors of $I_i$ have height $i$, by Theorem \ref{unmixed}. Since $\Ht(f_1, \ldots, f_i, f_{i+1}) = i+1$ by hypothesis, $f_{i+1}$ is not contained in any prime divisor of $I_i$, so $f_{i+1}$ is $(S/I_i)$-regular. Using this argument for $i = 1, 2, \ldots, t-1$, we see that $f_1, \ldots, f_t$ is a regular sequence.
\end{proof}

\subsection{Notions of Non-Noetherian Cohen-Macaulayness}

Asgharzadeh and Tousi \cite{at} have identified seven notions of non-Noetherian Cohen-Macaulayness in the literature. We recall from Definition \ref{nnassoc}  that a prime ideal $P$ of a ring $R$ is a \emph{weak-Bourbaki associated (or weakly associated) prime} of an ideal $I$ if $P$ is minimal over $(I :_R a)$ for some $a \in R \setminus I$. One of the seven notions of non-Noetherian Cohen-Macaulayness features weak-Bourbaki associated primes.

\begin{de} \cite[Definitions 1 and 2]{H2} Let $I$ be a finitely generated ideal of $R$, and let $\mu(I)$ be the minimal number of elements needed to generate $I$. If for each $I$ with $\Ht (I) \geq \mu(I)$, the set of minimal prime ideals of $I$ coincides with the set of weak-Bourbaki associated primes of $I$, then we call such rings \emph{weak-Bourbaki unmixed}. If $R$ is weak-Bourbaki unmixed, we say that $R$ is Cohen-Macaulay in the sense of ``WB".
\end{de}

We present an example of a power series ring over a 0-dimensional ring that does not satisfy the ``WB" notion of non-Noetherian Cohen-Macaulayness.

\begin{ex}
	Consider the ring
	\[
		R = \mathbb{Q}[Y, \{ Z_i\}_{i=0}^\infty] / (Y^2, \{Z_i^2\}_{i=0}^\infty, \{YZ_i\}_{i=0}^\infty).
	\] 
	We first show that $R$ is a 0-dimensional local ring that is not SFT, hence non-Noetherian. 
	
	Put $S =  \mathbb{Q}[Y, \{ Z_i\}_{i=0}^\infty]$, and denote the image of $Y$ and $Z_i$ for each $i$ in $R$ by $y$ and $z_i$, respectively. If $P \in \spec R$, then we have $y \in P$ and $z_i \in P$ for each $i$ since $y^2 = z_i^2 = 0$. The ideal $Q = (Y, \{Z_i\}_{i=0}^\infty)S$ is a maximal ideal of $S$, and $P$ contains the image of $Q$. Hence, $R$ is a 0-dimensional local ring with maximal ideal $P = (y, \{z_i\}_{i=0}^\infty)R$. We note that $y \cdot P = 0$ since $y^2 = 0$ and $y z_i = 0$ in $R$ for each $i$.
	
	Next, we show that $P$ is a non-SFT ideal. Let $J$ be a finitely generated proper ideal of $R$, and let $k \in \mathbb{N}$. Since $J$ is finitely generated, there is some $s$ such that any nonzero product in the $z_t$'s with $t \geq s$ is not in $J$. Consider the element $f_k := z_s + z_{s+1} + \cdots + z_{s+k-1} \in P$. Then every term of $f_k^k$ will be 0 except for the term $k! \cdot z_s \cdot z_{s+1} \cdots z_{s+k-1}$. Since $z_s \cdot z_{s+1} \cdots z_{s+k-1} \notin J$, we have $f_k^k \notin J$. Hence, for any $k \in \mathbb{N}$, there is an element $f_k \in P$ such that $f_k^k \notin J$, and so $P$ is a non-SFT ideal, and $R$ is a non-SFT ring. 
	
	We show next that the power series ring $R \llbracket X \rrbracket$ fails to satisfy the ``WB" notion of non-Noetherian Cohen-Macaulayness. Consider the element $f := y \in R \llbracket X \rrbracket$, and let $g  \in P \llbracket X \rrbracket$. Every coefficient of $g$ is in $P$ and $y \cdot P = 0$, so we find that the product $f \cdot g$ is zero, and $f \cdot P \llbracket X \rrbracket = 0$. This means $P \llbracket X \rrbracket \subseteq (0: f)$. 
	On the other hand, if $h \in R \llbracket X \rrbracket \setminus P \llbracket X \rrbracket$, then $h$ has some unit coefficient, so $f \cdot h \neq 0$. Thus, we have $P \llbracket X \rrbracket = (0: f)$, and $P\llbracket X \rrbracket$ is a weak-Bourbaki associated prime ideal of the zero ideal $(0)$. Since $R$ is a 0-dimensional non-SFT local ring, we have $\Ht(P\llbracket X \rrbracket) = \infty$ by \cite[Corollary 32]{toankang}. Since every minimal prime divisor of the zero ideal has height 0, the set of minimal prime divisors of $(0)$ does not coincide with the set of weak-Bourbaki associated primes of $(0)$, and so $R \llbracket X \rrbracket$ fails to satisfy ``WB".
\end{ex}

For a semilocal 0-dimensional ring $R$, by Theorem \ref{unmixed} and Theorem \ref{gradeheight}, along with \cite[Proposition 2.3, Lemma 3.2]{at}, the ring $\mathcal{F} \llbracket X_1, \ldots, X_n \rrbracket$ satisfies all seven notions of non-Noetherian Cohen-Macaulayness given in \cite{at}. We record this fact to conclude.
\begin{cor}
	Let $R$ be a 0-dimensional semilocal ring, and let $R_0$ and $\mathcal{F}$ be as in Theorem \ref{unmixed}. Then $\mathcal{F}\llbracket X_1, \ldots, X_n \rrbracket$ satisfies each of the seven notions of non-Noetherian Cohen-Macaulayness presented in \cite{at}.
\end{cor}

We summarize the properties of $\mathcal{F}\llbracket X_1, \ldots, X_n\rrbracket$ for a 0-dimensional semilocal ring $R$ in Table \ref{summ}.

\newcolumntype{u}{>{\columncolor[RGB]{225,225,255}} m{2.1cm}}
\newcolumntype{s}{>{\columncolor[RGB]{225,225,255}} m{4.0cm}}
\newcolumntype{v}{>{\columncolor[RGB]{255,225,235}} m{4.2cm}}
 \newcolumntype{w}{>{\columncolor[RGB]{255,200,200}} m{2.8cm}}
  \newcolumntype{q}{m{2.0cm}}
\newcolumntype{x}{>{\columncolor[RGB]{255,225,235}} m{5.0cm}}

\begin{table}[!h] 
\begin{center}
\begin{tabular}{ q| s| v | w | } 

\cline{2-4}
 & \multicolumn{1}{c|}{\textbf{Noetherian Ring} $R$}
           & \multicolumn{2}{c|}{\textbf{Non-Noetherian Ring} $R$} \\
       & \centering $ R \llbracket X_1, \ldots, X_n \rrbracket$ 
         &  \centering $ R \llbracket X_1, \ldots, X_n \rrbracket $ 
         & \centering $\mathcal{F}\llbracket X_1, \ldots, X_n \rrbracket$  \tabularnewline \hline
      \cline{2-4}
        \centering  {  Dimension } &  \centering $n$ \cellcolor[RGB]{240,240,255}  &\quad \ \centering $n$ if $R$ is \emph{SFT} \newline {\small ($\infty$ otherwise) }\cellcolor[RGB]{255,240,248}  &  \centering $n$ \cellcolor[RGB]{255,240,240}  \tabularnewline \hline
        \centering  {  Flatness } & \centering $\quad$ Faithfully Flat $\quad$ over $R$ &  \centering  { Faithfully Flat over $R$} \newline  \hspace{0.5in} { if $R$ is \emph{coherent}} \hspace{-0.28in} \vspace{-0.03in} \newline  { \scriptsize (may not be flat otherwise) }  &  \centering Faithfully Flat over $R$  \tabularnewline  \hline 
       
       \centering \vspace{0.05in} \hspace{0.001in} CM \vspace{0.05in} & \centering  \vspace{0.05in} \hspace{0.001in} Yes \vspace{0.05in}   \cellcolor[RGB]{240,240,255}  &  \centering  \vspace{0.05in}  \hspace{0.001in} Not in general \vspace{0.05in}  \cellcolor[RGB]{255,240,248}  & \centering  \vspace{0.05in}  \hspace{0.001in} Yes  \vspace{0.05in}  \cellcolor[RGB]{255,240,240}  \tabularnewline  
\end{tabular}
\end{center}
\caption{Properties of power series rings over a 0-dimensional semilocal ring $R$ that is integral over an Artinian subring $R_0$.} \label{summ}
\end{table}

\section*{Acknowledgements}

The author would like to thank his advisor Bruce Olberding at New Mexico State University for his suggestions, generosity, and encouragement.

\input{biblio}


%

%



\end{document}

%% file: biblio.tex


%
%
%
%

%% file: 20251008_Power_Series_Rings_Sayanagi.bbl
\begin{thebibliography}{99}

\bibitem{anderson}
F.W. Anderson and K.R. Fuller, \emph{Rings and categories of modules}, Springer-Verlag New York, 1992.

\bibitem{Arnold1973}
J.T. Arnold, \emph{Krull dimension in power series rings}, Trans. Amer. Math. Soc. \textbf{177} (1973), 299-304.

\bibitem{at}
M. Asgharzadeh and M. Tousi, \emph{On the notion of Cohen–Macaulayness for non-Noetherian rings}, J. of Algebra \textbf{322} (2009), no. 7, 2297-2320.

\bibitem{chase}
S.U. Chase, \emph{Direct products of modules}, Trans. Amer. Math. Soc. \textbf{97} (1960), no. 3, 457-473.

\bibitem{ccd}
J.T. Condo, J. Coykendall, and D.E. Dobbs, \emph{Formal power series rings over zero-dimensional SFT-rings}, Comm. Algebra \textbf{24} (1996), 2687-2698.

\bibitem{coykendall}
J. Coykendall, \emph{The SFT property does not imply finite dimension for power series rings}, J. of Algebra \text{256} (2002), 85-96.

\bibitem{fuchsheinzerolberding}
L. Fuchs, W. Heinzer, and B. Olberding, Maximal prime divisors in arithmetical rings, in \emph{Rings, modules, algebras, and abelian groups}, 189-203, Lecture Notes in Pure and Appl. Math., 236, Dekker, New York, 2004.

\bibitem{GH1992}
R. Gilmer and W. Heinzer, \emph{Products of commutative rings and zero-dimensionality}, Trans. Amer. Math. Soc. \textbf{331} (1992), 663-680.

\bibitem{GH1993}
R. Gilmer and W. Heinzer, \emph{Artinian subrings of a commutative ring}, Trans. Amer. Math. Soc. \textbf{336} (1993), no. 1, 295-310.

\bibitem{glaz}
S. Glaz, \emph{Homological dimensions of localizations of polynomial rings}, Zero-Dimensional Commutative Rings, Knoxville, TN, 1994, Lect. Notes Pure Appl. Math \textbf{171}, Marcel Dekker, New York (1995), 209-222.

\bibitem{H2}
T. D. Hamilton, \emph{Weak Bourbaki unmixed rings: A step towards non-Noetherian Cohen-Macaulayness}, Rocky Mountain J. Math. \textbf{34} (2004), no. 3, 963-977.

\bibitem{HM}
T. D. Hamilton and T. Marley, \emph{Non-Noetherian Cohen-Macaulay rings}, J. Algebra \textbf{307} (2007), 343-360.

\bibitem{loper}
B.G. Kang, K.A. Loper, T.G. Lucas, M.H. Park, and P.T. Toan, \emph{The Krull dimension of power series rings over non-SFT rings}, J. of Pure and Appl. Algebra \textbf{217} (2013), 254-258.

\bibitem{kaplansky}
I. Kaplansky, \emph{Commutative rings}, The University of Chicago Press, 1974.

\bibitem{magarian}
E.A. Magarian, \emph{Direct limits of power series rings}, Mathematica Scandinavica \textbf{31} (1972), no. 1, 103-110.

\bibitem{matsumura}
H. Matsumura, \emph{Commutative ring theory}, Cambridge University Press, 1986.

\bibitem{olb1}
B. Olberding, \emph{Height Theorems and Unmixedness for Finitely Generated Algebras over Zero-Dimensional Rings}, Proc. Amer. Math. Soc. \textbf{149} (2021), no. 11, 4515-4526. 

\bibitem{Seidenberg}
A. Seidenberg, \emph{A note on the dimension theory of rings}, Pacific J. Math. \textbf{3} (1953), 505-512.

\bibitem{toankang}
P.T. Toan and B.G. Kang, \emph{Krull dimension of power series rings}, J. Algebra \textbf{562} (2020), 306-322.



\end{thebibliography}
